\documentclass[12pt,a4paper,oneside]{amsart}

\usepackage{amsfonts, amsmath, amssymb, amsthm, amscd, hyperref}
\usepackage{anysize}
\usepackage{graphicx}

\newtheorem{theorem}{Theorem}[section]

\newtheorem{corollary}[theorem]{Corollary}
\newtheorem{proposition}[theorem]{Proposition}

\theoremstyle{definition}
\newtheorem{definition}[theorem]{Definition}

\theoremstyle{remark}
\newtheorem{remark}[theorem]{Remark}

\numberwithin{equation}{section}

\DeclareMathOperator{\volu}{vol}

\DeclareMathOperator{\conv}{conv}

\DeclareMathOperator{\inte}{int}

\renewcommand{\epsilon}{\varepsilon}

\begin{document}

\title[Shortest closed billiard trajectories\dots]{Shortest closed billiard trajectories in the plane and equality cases in 
Mahler's conjecture}

\author[A.~Balitskiy]{Alexey~Balitskiy}

\email{alexey\_m39@mail.ru}

\address{Dept. of Mathematics, Moscow Institute of Physics and Technology, Institutskiy per. 9, Dolgoprudny, Russia 141700}
\address{Institute for Information Transmission Problems RAS, Bolshoy Karetny per. 19, Moscow, Russia 127994}


\begin{abstract}
In this note we prove some Rogers--Shepard type inequalities for the lengths of shortest closed billiard trajectories, mostly in the planar case. We also establish some properties of closed billiard trajectories in Hanner polytopes, having some significance in the symplectic approach to the Mahler conjecture.
\end{abstract}

\maketitle

\section{Introduction}

In recent works~\cite{aao2012,aaok2013} the authors establish an interesting connection between billiards and some problems from symplectic geometry and convex geometry. Namely, with the help of the billiard approach they reduce the famous Mahler conjecture (see~\cite{ma1939,tao2007}) to Viterbo's conjecture from symplectic geometry.


Let an $n$-dimensional real vector space $V$ be endowed with a norm with unit ball $T^\circ$ (where $T^\circ \subset V$ is polar to a convex body $T \subset V^*$). We denote such norm by $\|\cdot\|_T$, mind the nonstandard notation and interchange of the original space and its dual. By definition, $\|q\|_T = \max\limits_{p \in T} \langle p, q \rangle$, where $\langle \cdot, \cdot \rangle : V^* \times V \rightarrow \mathbb{R}$ is the canonical bilinear form of the duality between $V$ and $V^*$. Here we assume that $T$ contains the origin (although this can be relaxed to some extent), but is not necessarily centrally symmetric. Therefore the norm may be non-symmetric, in general, $\|q\|_T \neq \|-q\|_T$. We call such norms \emph{flat Finsler norms} to distinguish them from the symmetric case.

The connection between symplectic capacities and billiards is that
$$
\xi_T(K) = c_{HZ}(K \times T),
$$
where $c_{HZ}(\cdot)$ stands for the so called Hofer--Zehnder symplectic capacity and
$\xi_T(K)$ denotes the length of the shortest closed billiard trajectory in a convex body $K \subset V$ with geometry of lengths given by a convex body $T \subset V^*$ and its norm $\|\cdot\|_T$ in $V$.

More precisely, we measure lengths in $V$ using the norm $\|\cdot\|_{T}$ and the billiard reflection rule is given by locally minimizing the length functional. We say that a polygonal line $q_{start} \rightarrow q_{refl} \rightarrow q_{end}$ (where $q_{refl} \in \partial K, q_{start} \in K, q_{end} \in K$) has a \emph{billiard reflection} at the point $q_{refl}$ if the functional
$$
\varphi(q) = \|q_{end} - q\|_T + \|q - q_{start}\|_T
$$
has a local minimum at the point $q = q_{refl}$ under the constraint $q \in \partial K$. If $q_{refl}$ belongs to the smooth piece of $\partial K$ we say that a \emph{classical} billiard reflection occurs.

In such a case one can rewrite the reflection rule in the differential form:
\begin{equation}
\label{equation:reflection}
p' - p = - \lambda n_K(q), \quad \lambda > 0.
\end{equation}
Here we define the momenta $p, p' \in V^*$ before and after the reflection so that a velocity $v = \dot q$ can be found by the formula ($d$ denotes the differential)
$$
v = d\|p\|_{T^\circ}, \quad v \in \partial T^\circ.
$$
In order to have such a formula, we add a requirement that $p$ belongs to the smooth piece of $\partial T$.
Also here we define the outer normal to the body $K$ at a point $q \in \partial K$ as
$$
n_K(q) = d\|q\|_{K^\circ}, \quad n_K(q) \in \partial K^\circ.
$$
In the rest of this section we assume the bodies $K$ and $T$ smooth so that velocities and normals are well-defined, although later we will be able to work with arbitrary convex bodies with some caution.

\begin{figure}
\centering
\includegraphics[width=1.0\textwidth]{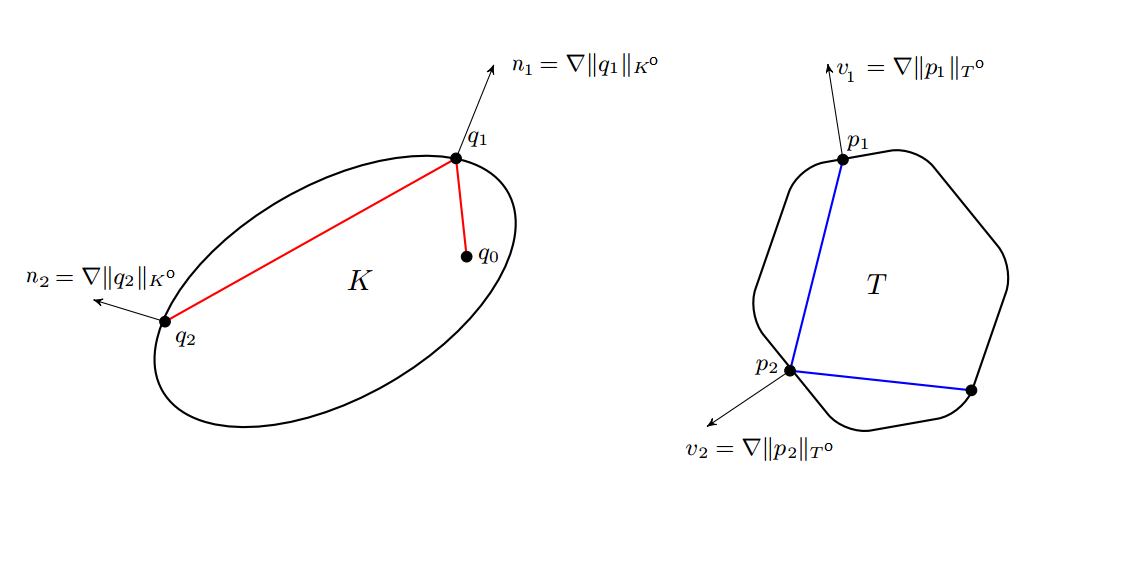}
\caption{Illustration to billiard reflection rule}
\label{picture:reflection}
\end{figure}

Following~\cite{abks2014}, we use Bezdeks' approach characterizing shortest closed billiard trajectories (see also the original paper~\cite{bb2009}). The approach is that instead of finding the shortest closed billiard trajectory in $K$ we can search for the shortest polygonal line that cannot be translated into $\inte K$. More formally, we define
$$
\mathcal P_m(K) = \{(q_1,\ldots, q_m) : \{q_1,\ldots, q_m\} \ \text{doesn't fit into}\ (\inte K + t)\ \text{with} \ t\in V \} =
$$
$$
= \{(q_1,\ldots, q_m) : \{q_1,\ldots, q_m\} \ \text{doesn't fit into}\ (\alpha K+t)\ \text{with}\ \alpha\in (0,1),\ t\in V \}
$$
and
$$
\xi_T(K) = \min_{Q \in \mathcal Q_T(K)} \ell_T(Q),
$$
where $Q = (q_1,\ldots, q_m),\ m \ge 2,$ ranges over the set $\mathcal Q_T(K)$ of all closed billiard trajectories in $K$ with geometry defined by $T$. (Here we denote the length $\ell_T (q_1,\ldots, q_m) = \sum_{i=1}^m \|q_{i+1} - q_i\|_T$.)

The main result of~\cite{bb2009}, revised and slightly generalized in~\cite{abks2014}, states the following.

\begin{theorem}
\label{theorem:bezdeks}
For any smooth convex bodies $K \subset V, T \subset V^*$ containing the origins of $V, V^*$ in their interiors an equality holds:
$$
\xi_T(K) = \min_{m\ge 2} \min_{Q\in \mathcal P_m(K)} \ell_T(Q);
$$
and furthermore, the minimum is attained at $m\le n + 1$.
\end{theorem}

\begin{remark}
\label{remark:xidef}
The right hand side of the above formula is defined without any assumption
on the smoothness of $K$ and $T$. In what follows we use it as the definition of $\xi_T(K)$ even
when neither $K$, nor $T$ are smooth.
\end{remark}

Also in what follows we use the notion of the \emph{width} of a body $K \subset V$:
$$
w_T(K) = \min_{p \in V^*, \| p \|_{T^{\circ}} = 1} (\max_{q \in K} \langle p,q \rangle - \min_{q \in K} \langle p,q \rangle).
$$

Theorem~\ref{theorem:bezdeks} allows to give elementary proofs (see~\cite{abks2014}) for some properties of the value $\xi_T(K)$. Originally they were proved in~\cite{aao2012} using non-trivial symplectic techniques.

\begin{itemize}
\item
Monotonicity:
\begin{equation}
\label{equation:monotonicity}
\xi_T(K) \le \xi_T(L)\ \text{when}\ K\subseteq L;
\end{equation}

\item
Symmetry:
\begin{equation}
\label{equation:symmetry}
\xi_T(K) = \xi_K(T);
\end{equation}

\item
Brunn--Minkowski type inequality:
\begin{equation}
\label{equation:bm}
\xi_T(K+L) \ge \xi_T(K) + \xi_T(L).
\end{equation}
\end{itemize}

The structure of this paper is as follows. Section~\ref{rogshep} is devoted to some inequalities of the type 
$\xi_{T}(K-K) \le c \xi_{T}(K)$, mainly in the plane. In section~\ref{traeceq} we study some examples when all billiard trajectories have the same length. In section~\ref{constwid} we give one more proof for the theorem about shortest billiard trajectories in constant width bodies in the plane.

\textbf{Acknowledgments.}

The author is grateful to Roman~Karasev for constant attention to this work and to Yaron~Ostrover for useful remarks.

\section{Rogers--Shepard type inequalities}
\label{rogshep}

The famous Rogers--Shepard inequality estimates the volume of the Minkowski difference $K-K$ from above in terms of the volume of the original convex body $K$, that is $\volu(K-K) \le \binom{2n}{n}\volu K$, where $n=\dim K$.

The natural question in our setting is to estimate $\xi_{T}(K-K)$ from above in terms of $\xi_{T}(K)$.

First, we consider the general case of Finsler possibly non-symmetric norms.

\begin{theorem}
\label{theorem:rogshepfinsl}
For convex bodies $K, T$ in $\mathbb{R}^n$ an estimate holds:
$$
\xi_{T}(K-K) \le (n+1) \xi_{T}(K).
$$
In $\mathbb{R}^2$ this estimate is sharp.
\end{theorem}

\begin{proof}

It can be easily deduced from Helly's theorem that for any body $K \subset \mathbb{R}^n$ the body $-K$ can be translated into $nK$. Moreover, if the center of gravity of $K$ is located at the origin then $-K \subset nK$. Thus $K-K$ can be translated into $(n+1)K$ and monotonicity~\ref{equation:monotonicity} yields the required estimate.

Now we show the sharpness for the case of an equilateral triangle $K$ and $T = K^{\circ}$. In other words, we check that in the hexagon $K-K$ there are no billiard trajectories shorter than in the triangle $3K$ provided we measure lengths in the norm with the unit body $3K$ (we inflate the unit body for our convenience, using the homogeneity of $\xi_T(K)$ in $T$). The length and the form of the trajectory delivering $\xi_{(3K)^{\circ}}(3K) = 3$ are known from the proof of~\cite[Theorem~3.1]{abks2014}.

\begin{figure}
\centering
\includegraphics[width=0.5\textwidth]{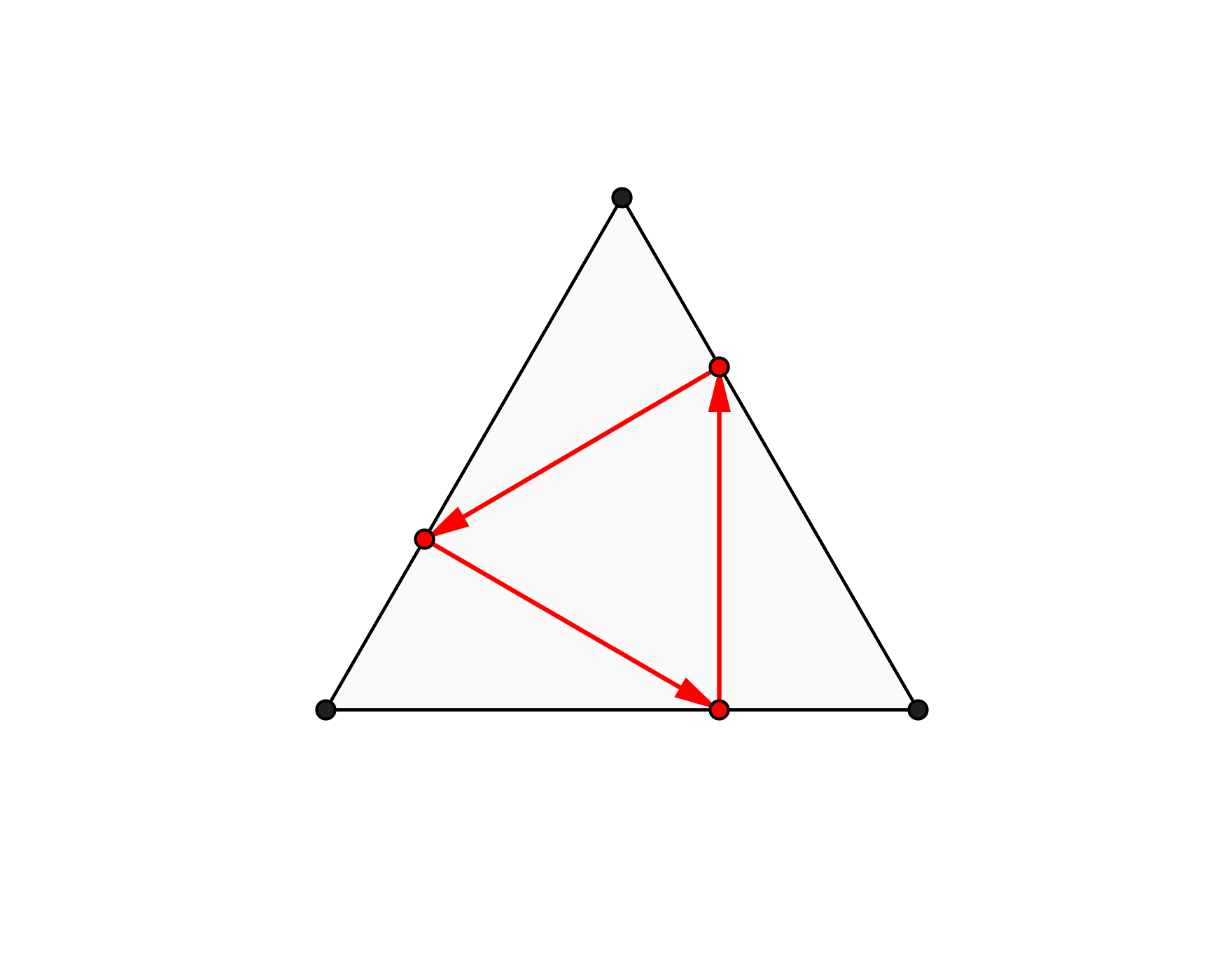}
\caption{The form of the trajectory delivering $\xi_{(3K)^{\circ}}(3K) = 3$}
\label{picture:triangle}
\end{figure}

Consider all possible candidates for the role of the shortest trajectory in $K-K$. Theorem~\ref{theorem:bezdeks} implies that it can only be 2- or 3-periodic polygonal line that cannot be translated into $\inte K$.

\begin{figure}
\centering
\includegraphics[width=1.0\textwidth]{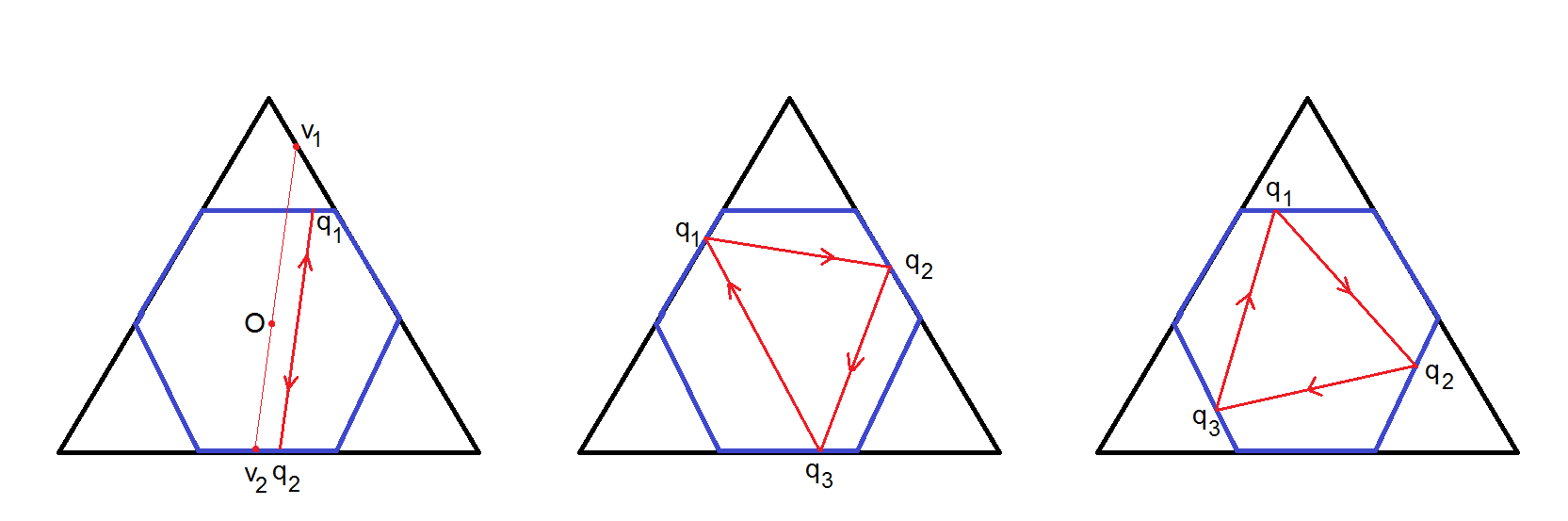}
\caption{Candidates for the role of the shortest trajectory in $K-K$}
\label{picture:rogshepfinsl}
\end{figure}

\begin{enumerate}
\item The case of 2-periodic trajectory. Since it cannot be translated into $\inte K$, the both reflections must occur at opposite hexagon sides like in Figure~\ref{picture:rogshepfinsl}.
In the notation from the figure we have
$$
\ell_{(3K)^{\circ}}(Q) = \frac{|q_1-q_2|}{|v_1|} + \frac{|q_1-q_2|}{|v_2|} \ge 1 + 2 = \xi_{(3K)^{\circ}}(3K)
$$
(where $|\cdot|$ denotes the usual Euclidean norm).

\item The case of 3-periodic trajectory. Three reflections will occur at the points of some three sides. If, among these sides, there are two opposite then the problem is reduced to the previous case. Otherwise there are two possibilities depicted in Figure~\ref{picture:rogshepfinsl}. The first possibility of the trajectory also represents the trajectory in the triangle $3K$ so it cannot be shorter than $\xi_{(3K)^{\circ}}(3K)$. The second possibility can be easily reduced to the first one: One should consider the reversed version of the trajectory and measure its length in the norm with the unit body $-3K$ (it will be the same as the length of the direct version in the norm with the unit body $3K$). After such a modification we are in the setting of the first possibility.
\end{enumerate}
\end{proof}

In this proof the symplectic capacity of $(K-K) \times K^{\circ}$ was investigated. The similar question --- about the symplectic capacity of $K \times (K-K)^{\circ}$ --- was considered in~\cite{akp2014} (where an estimate from below $1+1/n$ is established) and in~\cite{nir} (where its sharpness is shown).

Now we turn to the Euclidean case when the norm in $V = \mathbb{R}^n$ is specified by the standard Euclidean ball, which we denote by $B$ from now on.

\begin{theorem}
\label{theorem:rogshepeucl}
For a convex body $K \subset \mathbb{R}^2$ a sharp estimate holds:
$$
\xi_{B}(K-K) \le \frac{4}{\sqrt{3}} \xi_{B}(K).
$$
\end{theorem}

\begin{proof}

In~\cite{ghomi2004} it was shown that for the a body $L$ that is centrally symmetric with respect to the origin an equality holds:
$\xi_B(L) = 2 w_B(L)$ (another elementary proof can be found in~\cite{abks2014}). Also it is easy to see that $w_B(K-K) = 2w_B(K)$.

Therefore $\xi_{B}(K-K) = 4w_B(K)$ and the required estimate can be reformulated in the following way:

$$
\xi_{B}(K) \ge \sqrt{3} w_B(K).
$$

Assume the contrary: For a body $K$ in the plane it happens that $\xi_{B}(K) < \sqrt{3} w_B(K)$.
In particular, the shortest billiard trajectory is 3-periodic, since the shortest 2-periodic trajectory has the length $2w_B(K) > \sqrt{3} w_B(K)$. (In the case of non-smooth $K$ we understand the shortest billiard trajectory in the sense of~\ref{remark:xidef}.)

Let $A,B,C$ be the vertices of such a trajectory. Consider the supporting lines for $K$ at the points $A,B,C$. (In the case of non-smooth $K$ when supporting lines cannot be determined uniquely we can choose them to be orthogonal to bisectors of $\angle A, \angle B, \angle C$.) These lines form a $\triangle A'B'C'$, where $A,B,C$ are altitudes' bases. $K \subseteq \triangle A'B'C'$ and monotonicity of the width yields $w_B(K) \le w_B(\triangle A'B'C')$; but in $\triangle A'B'C'$ the 3-periodic trajectory has the same length as in $K$. Thus it is sufficient to obtain the contradiction with an inequality $\xi_{B}(\triangle A'B'C') < \sqrt{3} w_B(\triangle A'B'C')$.

\begin{figure}
\centering
\includegraphics[width=0.7\textwidth]{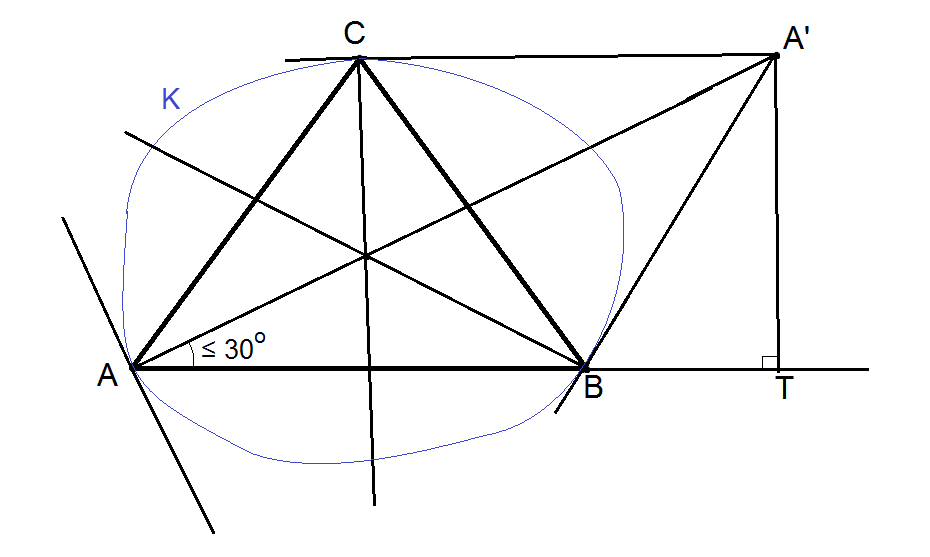}
\caption{Illustration to the proof of theorem~\ref{theorem:rogshepeucl}}
\label{picture:rogshepeucl}
\end{figure}

Let $\angle A$ be the minimal angle in $\triangle ABC$. Then $\angle A'$ is the maximal in $\triangle A'B'C'$ and the width of $\triangle A'B'C'$ equals the length of the altitude from $A'$, i.e. $w_B(\triangle A'B'C') = |AA'|$. Let $T$ be the base of the perpendicular to the line $(AB)$ through $A'$. $A'$ is the center of the excircle $\triangle ABC$; hence, $|AT| = \frac12 (|AB|+|BC|+|CA|) = \frac12 \xi_{B}(\triangle A'B'C')$. But $\angle TAA' \le 30^{\circ}$; hence $\dfrac{\xi_{B}(\triangle A'B'C')}{w_B(\triangle A'B'C')} = 2 \dfrac{|AT|}{|AA'|} \ge 2 \sin 30^{\circ} = \sqrt{3}$. This proves the required estimate.

To show the sharpness, consider the equilateral triangle: The ratio $\dfrac{\xi_B(\cdot)}{w_B(\cdot)}$ (i.e. the ratio between perimeter of the midpoint triangle and the altitude) indeed equals~$\sqrt{3}$.
\end{proof}

\section{When all the billiard trajectories have the same length}
\label{traeceq}

In the proof of Theorem~\ref{theorem:rogshepfinsl} we have seen that the configuration $(K-K) \times K^{\circ}$, where $K$ is an equilateral triangle, has a peculiar property: Both shortest 2-periodic polygonal line and shortest 3-periodic polygonal line that cannot be translated into $\inte(K-K)$ have the same length equal to 9. But both of these polygonal lines aren't classical billiard trajectories because they have either coordinates or momenta in the vertices of polygons $K-K$ and $K^{\circ}$.

Further we consider only classical billiard trajectories, i.e. require $q_i$ and $p_i$ must belong to the smooth pieces of $\partial(K-K)$ and $\partial(K^{\circ})$. Straightforward consideration shows that there are no 2-periodic and 3-periodic in this configuration. An example of a 4-periodic trajectory is given in Figure~\ref{picture:traecinhex}.

\begin{figure}
\centering
\includegraphics[width=0.85\textwidth]{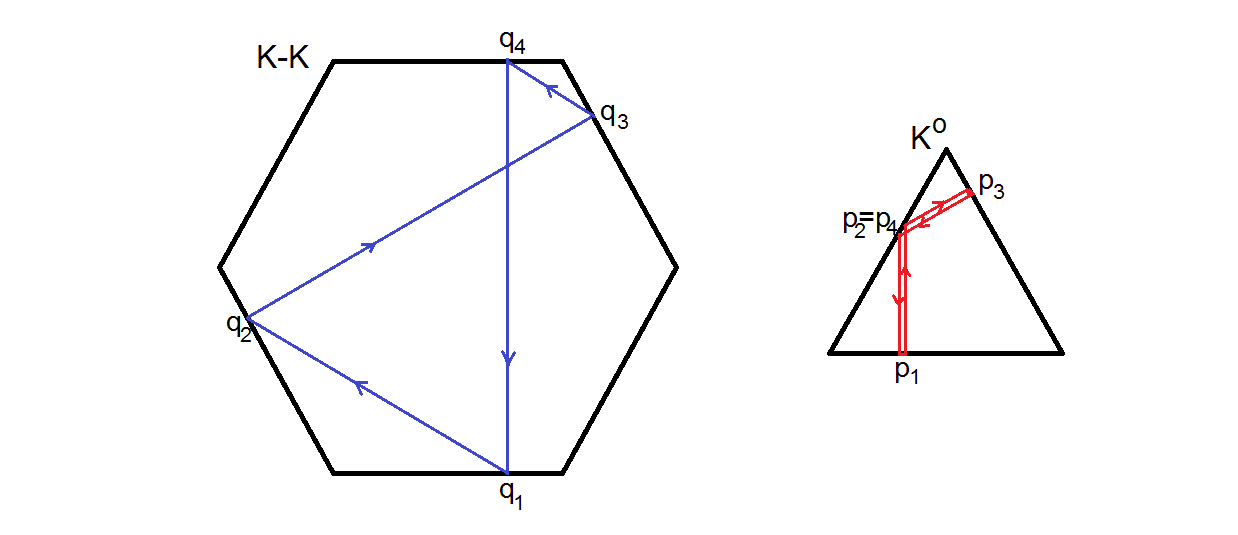}
\caption{4-periodic classical billiard trajectory}
\label{picture:traecinhex}
\end{figure}

It can be checked easily that all classical trajectories in this configuration have the same form (up to rotations, symmetries and translations of the starting point along the sides): Any initial coordinate $q_1$ and the initial momentum $p_1$ (chosen in such a way that the normal $-n_{K-K}(q_1)$ drawn from the point $p_1$ looks inside $K^{\circ}$) can be extended to a 4-periodic classical billiard trajectory. Its length equals 9 (the same as for the shortest 2- and 3-periodic polygonal lines that cannot be translated into $\inte(K-K)$).
If we forbid trajectories that pass the same route multiply (i.e. consider only \emph{simple} trajectories) then the following observation holds:

\begin{proposition}
Any simple classical billiard trajectory in the configuration $(K-K) \times K^{\circ}$ (where $K$ is an equilateral triangle with the center at the plane origin) is 4-periodic and has the length 9.

Also, arbitrarily close to any point of $\partial((K-K) \times K^{\circ})$ there passes a certain shortest billiard trajectory.
\end{proposition}

Such a situation is supposed to take place in the cases of equality of Mahler's conjecture (see~\cite{apb2012}). Further we establish the similar result for the entire class of bodies which are famous for delivering the equality to Mahler's conjecture:

\begin{definition}
Let $\|\cdot\|^{(1)}$ be a norm in a real vector space $V_1 = \mathbb{R}^k$ with a unit body $K$ and $\|\cdot\|^{(2)}$ be a norm in a real vector space $V_2 = \mathbb{R}^l$ with a unit body $L$. Let's define the \emph{direct sum of the bodies} $K$ and $L$ according to the norms $\ell_1$ and $\ell_{\infty}$ in the following way:

\begin{itemize}
\item we define $K \oplus_1 L$ as a unit body of a norm $\|\cdot\|^{(1)} + \|\cdot\|^{(2)}$ in the space $V_1 \oplus_1 V_2 = \mathbb{R}^{k+l}$;

\item we define $K \oplus_{\infty} L$ as a unit body of a norm $\max \{ \|\cdot\|^{(1)}, \|\cdot\|^{(2)}\}$ in the space $V_1 \oplus_{\infty} V_2 = \mathbb{R}^{k+l}$.
\end{itemize}
\end{definition}

\begin{definition}
The \emph{Hanner polytopes} are the bodies constructed in the following way:

\begin{itemize}
\item a segment $[-1,1]$ is the 1-dimensional Hanner polytope;

\item if $K \subset \mathbb{R}^k$ and $L \subset \mathbb{R}^l$ are the Hanner polytopes, then the bodies $K \oplus_1 L$ and $K \oplus_{\infty} L$ are the Hanner polytopes $\mathbb{R}^{k+l}$.
\end{itemize}

Getting closer to Mahler's conjecture, we need to consider a convex body together with its polar. The pairs consisting of a Hanner polytope and its polar are constructed inductively in the following way:

\begin{itemize}
\item in $\mathbb{R}^1$ it is $[-1,1] \times [-1,1]$;

\item using the pairs $K \times K^{\circ}$ and $L \times L^{\circ}$ one can construct the pairs $K \oplus_1 L \times K^{\circ} \oplus_{\infty} L^{\circ}$ and $K \oplus_{\infty} L \times K^{\circ} \oplus_1 L^{\circ}$.
\end{itemize}
\end{definition}

For what follows, it might be helpful to explain what are the facets of polytopes $K \oplus_{\infty} L$ and $K \oplus_1 L$ in terms of the facets of $K, L$:

\begin{itemize}
\item
a facet of $K \oplus_{\infty} L$ can be represented in the form $F_K \times L$ or $K \times F_L$ (here $F_M$ stands for a facet of $M$).
\item
a facet of $K \oplus_1 L$ can be represented in the form $\conv (F_K \times 0_l \cup 0_k \times F_L)$ (here $0_m$ stands for the origin in $\mathbb{R}^m$).
\end{itemize}

\begin{theorem}
\label{theorem:hanner}
Any simple classical billiard trajectory in a Hanner polytope $H \subset \mathbb{R}^n$ with geometry specified by its polar $H^{\circ}$

\begin{enumerate}
\item has length 4;
\item is $2n$-periodic (bounces $2n$ times);
\item is centrally symmetric with respect to the origin (i.e. $q_i = -q_{i+n}$) and the corresponding momenta trajectory is also centrally symmetric with respect to the origin (i.e $p_i = -p_{i+n}$).
\end{enumerate}
\end{theorem}

\begin{proof}
The 1-dimensional case is trivial.

Let $K \subset \mathbb{R}^k$ and $L \subset \mathbb{R}^l$ be Hanner polylopes and suppose the statement is proven for $K \times K^{\circ}$ and $L \times L^{\circ}$. We are to deduce the statement for the configuration $K \oplus_1 L \times K^{\circ} \oplus_{\infty} L^{\circ}$. (The symmetry property~\ref{equation:symmetry} allows us not to check the configuration $K \oplus_{\infty} L \times K^{\circ} \oplus_1 L^{\circ}$.)

Let $q_1 \rightarrow q_2 \rightarrow \ldots \rightarrow q_m \rightarrow q_1$ be coordinates of nodes in the trajectory (all of them lie in the relative interior of facets of $K \oplus_1 L$) and let $p_1 \rightarrow p_2 \rightarrow \ldots \rightarrow p_m \rightarrow p_1$ be corresponding momenta (they lie in the relative interior of facets of $K^{\circ} \oplus_{\infty} L^{\circ}$).

We write $q_i = (q_i^1, q_i^2), p_i = (p_i^1, p_i^2)$, decomposing $(k+l)$-dimensional coordinates into $k$-dimensional and $l$-dimensional components.

Let $\|q_1^1\|_{K^{\circ}} = \alpha$, $\|q_1^2\|_{L^{\circ}} = 1 - \alpha$.

The difference $q_2 - q_1$ is proportional to the normal $n_{K^{\circ} \oplus_{\infty} L^{\circ}}(p_2)$ which has the form $(0_k, n_{L^{\circ}}(p_2^2))$ or $(n_{K^{\circ}}(p_2^1), 0_l)$ (due to the form of facets of $K^{\circ} \oplus_{\infty} L^{\circ}$). If the first case $q_2^1 = q_1^1$ (we say that \emph{first type change} occurred), in the second case $q_2^2 = q_1^2$ (\emph{second type change}). In both cases  $\|q_2^1\|_{K^{\circ}} = \alpha$, $\|q_2^2\|_{L^{\circ}} = 1 - \alpha$ and we can continue in the similar way: $\|q_i^1\|_{K^{\circ}} = \alpha$, $\|q_i^2\|_{L^{\circ}} = 1 - \alpha$.

The difference $p_2 - p_1$ is proportional to the normal $n_{K \oplus_1 L}(q_1)$ which has the form $(n_{K}(q_1^1), n_{L}(q_1^2))$ where $q_1 = (\alpha q_1^1, (1 - \alpha) q_1^2)$, $q_1^1 \in F_K, q_1^2 \in F_L$, $0 < \alpha < 1$ (due to the form of facets of $K \oplus_1 L$). We say that the momentum $p_i$ has the \emph{first type} if it belongs to a facet of the form $F_{K^{\circ}} \times L^{\circ}$ and has the \emph{second type} if it belongs to a facet of the form $K^{\circ} \times F_{L^{\circ}}$.

Note that the momentum of a certain type leads to the coordinate change of the same type.

Now we do with the polygonal lines $q_1^1 \rightarrow q_2^1 \rightarrow \ldots \rightarrow q_m^1 \rightarrow q_1^1$ and $p_1^1 \rightarrow p_2^1 \rightarrow \ldots \rightarrow p_m^1 \rightarrow p_1^1$ the following operations. If in the coordinate polygonal line there are neighboring vertices that coincide, let's contract this pair to the single vertex removing the corresponding momentum. Repeat this unless the coordinate polygonal line has no neighboring coincidences. (We continue to use the notation $q_1^1 \rightarrow \ldots \rightarrow q_1^1$ and $p_1^1 \rightarrow \ldots \rightarrow p_1^1$ for thus obtained polygonal lines.)

\begin{figure}
\centering
\includegraphics[width=6in]{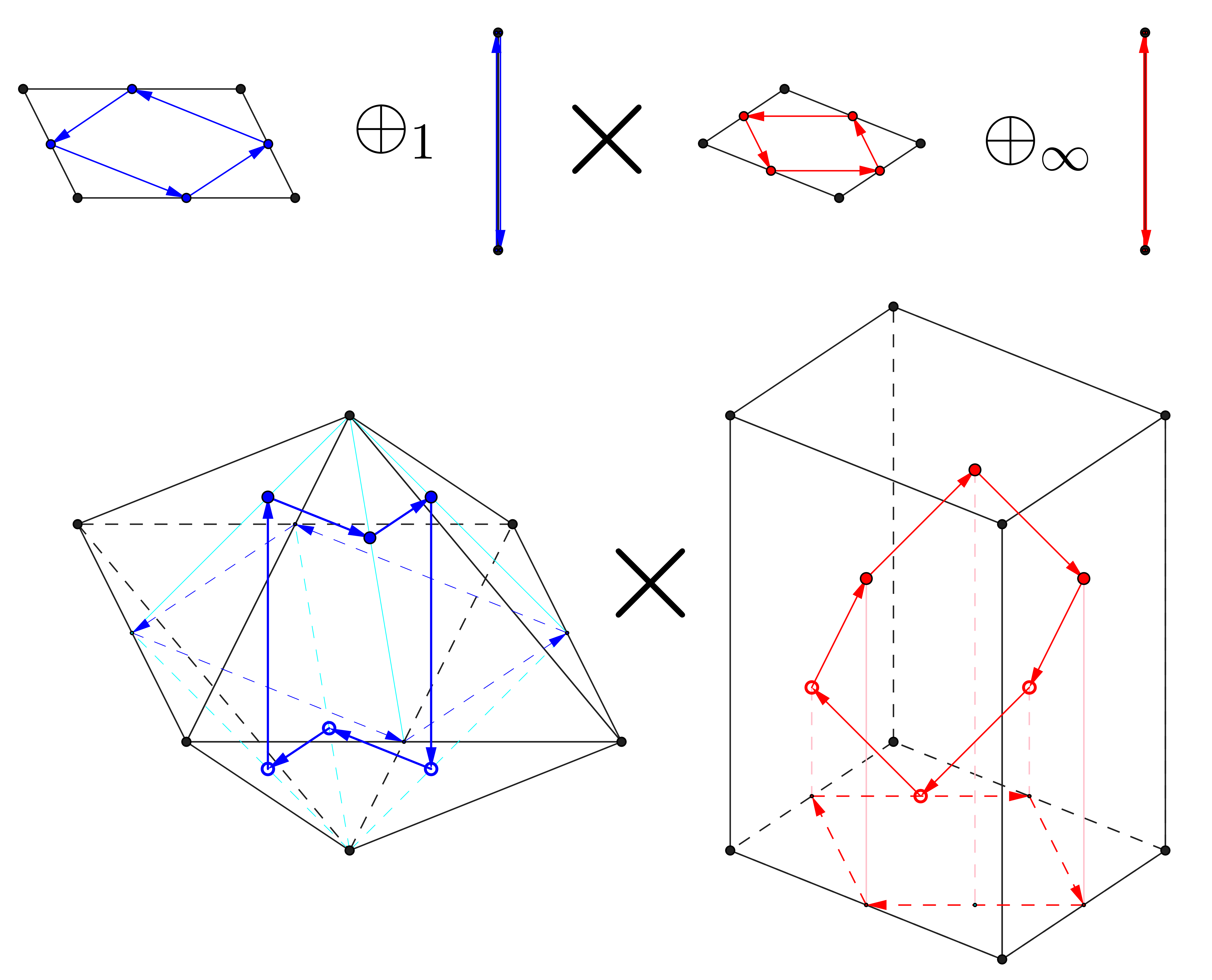}
\caption{Decomposition of a billiard trajectory in Hanner polytope}
\label{picture:hanner}
\end{figure}

We claim that after such a transformation we obtain a simple classical billiard trajectory in the body $\alpha K$ with geometry specified by $K^{\circ}$. As was shown before, $\|q_i^1\|_{K^{\circ}} = \alpha$, so $q_i^1 \in \partial(\alpha K)$. Moreover, $q_i^1$ lie in the relative interior of facets since the same holds for $q_i \in \partial (K \times L)$. The momenta $p_i^1$, that was not removed in the process, are obtained from the momenta $p_i$ of the first type; hence $p_i^1$ also lie on the boundary of $K^{\circ}$ in the relative interior of facets. Also it follows from all said above that the difference $q_{i}^1 - q_{i-1}^1$ is proportional to the normal $n_{K^{\circ}}(p_i)$, and the difference $p_{i+1} - p_i$ is proportional to the normal $n_{K}(q_i)$. It proves that the constructed polygonal line is indeed a classical billiard trajectory in the body $\alpha K$ with geometry specified by $K^{\circ}$. We construct the polygonal lines $q_1^2 \rightarrow \ldots \rightarrow q_1^2$ and $p_1^2 \rightarrow \ldots \rightarrow p_1^2$ in the similar manner. It remains to check its simplicity.

Let us examine carefully the change in the momenta. The geometry of $K^{\circ} \oplus_{\infty} L^{\circ}$ implies the following: If we consider the piece of the trajectory from the point $p_1$ to a certain point $\hat{p}$ and also consider its projections on the first $k$ coordinates and on the last $l$ coordinates then the length of this piece (measured with the norm with unit body $K^{\circ} \oplus_{\infty} L^{\circ}$) is equal to each of the projections' lengths (measured with the norm unit bodies $K^{\circ}$ and $L^{\circ}$). Apply this observation to the entire momentum polygonal line in $K^{\circ} \oplus_{\infty} L^{\circ}$ and note that $p_1^1 \rightarrow \ldots \rightarrow p_1^1$ is its projection on the first $k$ coordinates. Indeed, the removed momenta of the second type would change nothing in the projections: If in the chain $p_{i-1} \rightarrow p_{i} \rightarrow p_{i+1}$ the middle momentum is of the second type it means that the projection $p_{i}^1$ lies inside the segment $[p_{i-1}^1, p_{i+1}^1]$. Similarly, $p_1^2 \rightarrow \ldots \rightarrow p_1^2$ is the projection of the entire momentum polygonal line on the last $l$ coordinates.

So, assume that the polygonal line $q_1^1 \rightarrow \ldots \rightarrow q_1^1$ bypasses the same route $t$ times, then induction hypothesis implies that its length equals $4t\alpha$ (if the norm unit body is $K \oplus_1 L$). Then symmetry~\ref{equation:symmetry} yields that the length of $p_1^1 \rightarrow \ldots \rightarrow p_1^1$ also equals $4t\alpha$ (if the norm unit body is $K^{\circ} \oplus_{\infty} L^{\circ}$); hence the source momentum polygonal line $p_1 \rightarrow \ldots \rightarrow p_1$ has the length $4t\alpha$ and other its projection $p_1^2 \rightarrow \ldots \rightarrow p_1^2$ has the length $4t\alpha$ too. Therefore, the both projection were passed $t$ times and the source polygonal line was passed $t$ times also. The source polygonal line was chosen simple, hence $t=1$. That's why the length $q_1^1 \rightarrow \ldots \rightarrow q_1^1$ equals $4\alpha$.

Similarly, the length of $q_1^2 \rightarrow q_2^2 \rightarrow \ldots \rightarrow q_m^2 \rightarrow q_1^2$ equals $4(1-\alpha)$. Therefore, the length of $q_1 \rightarrow q_2 \rightarrow \ldots \rightarrow q_m \rightarrow q_1$ equals 4, as required.

Further, by induction hypothesis there are $2k$ momenta of the first type and $2l$ momenta of the second type in the polygonal line $p_1 \rightarrow p_2 \rightarrow \ldots \rightarrow p_1$. Therefore, the period of the simple billiard trajectory in the configuration $K \oplus_1 L \times K^{\circ} \oplus_{\infty} L^{\circ}$ equals exactly $2(k+l)$.

At last, to show central symmetry we examine momenta behavior again. As was shown before, the length $q_1 \rightarrow \ldots \rightarrow q_1$ equals 4, and the length $p_1 \rightarrow \ldots \rightarrow p_1$ also equals 4. Recall that the length of the piece of momentum trajectory from $p_1$ to some point $\hat{p}$ is equal to the lengths of the both projections on the first $k$ coordinates and on the last $l$ coordinates. One can set $\hat{p} = p_{k+l+1}$ and see that this observation together with induction hypothesis implies $p_1 = -p_{k+l+1}$. This reasoning can be applied starting from any point $p_i$ so central symmetry of the momentum trajectory is proven. Also, the arguments above imply that the momenta $p_i$ and $p_{i+n}$ have the same type so in the procedure of the ``thinning'' of $q_1^1 \rightarrow \ldots \rightarrow q_1^1$ the contracted links can be divided into pairs (on the distance of $n$ links between each other). The similar situation is with the ``thinning'' of $q_1^2 \rightarrow \ldots \rightarrow q_1^2$. This, together with the induction hypothesis, proves the central symmetry of the coordinate trajectory.
\end{proof}

In the argument above the trajectory in the configuration $K \oplus_1 L \times K^{\circ} \oplus_{\infty} L^{\circ}$ was decomposed in some way into two trajectories in the configurations $K \times K^{\circ}$ and $L \times L^{\circ}$. The following proposition reverses this decomposition showing that in some sense there are ``many'' trajectories in Hanner polytope.

\begin{theorem}
\label{theorem:hannervariety}
Let $K \subset \mathbb{R}^k$ and $L \subset \mathbb{R}^l$ be Hanner polytopes, and let $Q^1 = (q_1^1, \ldots, q_{2k}^1)$ with momenta $p_1^1 \rightarrow \ldots \rightarrow p_1^1$ and $Q^2 = (q_1^2, \ldots, q_{2l}^2)$ with momenta $p_1^2 \rightarrow \ldots \rightarrow p_1^2$ be simple classical billiard trajectories in the pairs $K \times K^{\circ}$ and $L \times L^{\circ}$. We consider the Hanner polytope $H = K \oplus_1 L$ (the first $k$ coordinates correspond to the summand $K$ and the last $l$ coordinates correspond to the summand $L$) with its polar $H^{\circ} = K^{\circ} \oplus_{\infty} L^{\circ}$. Let an initial coordinate $q_1 = (\alpha q_1^1, (1-\alpha) q_1^2)$ and an initial momentum $p_1 = (\beta p_1^1 + (1-\beta)p_2^1, p_1^2)$ be specified.

Then for any $\alpha \in (0,1)$ and for almost any $\beta \in (0,1)$ an extension of this initial position according to billiard rules in $H \times H^{\circ}$ gives a simple classical billiard trajectory satisfying the following properties:

\begin{enumerate}
\item The projection of its coordinate polygonal line on the first $k$ coordinates is congruent to $\alpha Q^1$, and the projection on the last $l$ coordinates is congruent to $(1-\alpha) Q^2$;
\item The projection of its momentum polygonal line on the first $k$ coordinates coincides with the momentum polygonal line of the trajectory $Q^1$, and the same on the last $l$ coordinates and the momentum polygonal line of the trajectory $Q^2$.
\end{enumerate}
\end{theorem}

\begin{proof}

A potential trajectory in $H \times H^{\circ}$, if exists, generates two trajectories in $K \times K^{\circ}$ and in $L \times L^{\circ}$ as we have seen before. Such a trajectory satisfies the property: The piece of momentum polygonal line from $p_1$ to $\hat{p}$ has the length equal to the lengths of its projections on the first $k$ and the last $l$ coordinates. Thus, in order to have a classical billiard trajectory, we must require the following from the initial position: There must not exist such a time moment, when the both momentum trajectories projection are in the nodes. More formally: we forbid all such $\beta$ for whose there exists $0<l<4$ such that pieces of both momentum polygonal line projections of the length $l$, starting at the points $\beta p_1^1 + (1-\beta)p_2^1$ and $p_1^2$, finish at the nodes. In more detail: for each node $p_i^2$ one can consider $l$ equal to the length of the piece $p_1^2 \rightarrow \ldots \rightarrow p_i^2$, and for every such $l$ there exist a finite number of forbidden $\beta$'s (corresponding to the coincidences between $p_i^2$ arrival time and time of arrival to certain node of the first polygonal line). Therefore there will be forbidden only a finite number of $\beta$'s, and we claim that any other $\beta$ is appropriate.

Let us demonstrate how it can be checked. Fix arbitrary $\alpha$ and arbitrary unforbidden $\beta$. The initial position $(q_1, p_1)$ allows us to uniquely determine the momentum $p_2$: We must choose such $\lambda > 0$ that $p_2 = p_1 - \lambda n_{K \oplus_1 L}(q_1)$ belongs to $\partial H^{\circ}$.

From the geometric point of view the expression $p_1 - \lambda n_{K \oplus_1 L}(q_1) = (\beta p_1^1 + (1-\beta)p_2^1 - \lambda n_{K}(q_1^1), p_1^2 - \lambda n_{L}(q_1^2))$ for small $\lambda$ means the shift towards $p_2^1$ in the first projection and the shift towards $p_2^2$ in the second projection. Now we grow $\lambda$ until we first reach such $\lambda$, for which either $\beta p_1^1 + (1-\beta)p_2^1 - \lambda n_{K}(q_1^1) \in \partial K^{\circ}$, or $p_1^2 - \lambda n_{L}(q_1^2) \in \partial L^{\circ}$. Both these events cannot occur simultaneously (since $\beta$ is chosen unforbidden).

If $\beta p_1^1 + (1-\beta)p_2^1 - \lambda n_{K}(q_1^1) \in \partial K^{\circ}$ is the first to occur then $p_2 = (p_2^1, \beta' p_1^2 + (1-\beta') p_2^2)$, where $\beta' \in (0,1)$ is chosen such that $p_1^2 - \lambda n_{L}(q_1^2) = \beta' p_1^2 + (1-\beta') p_2^2$. In terms of the previous proof $p_2$ has the second type.

If $p_1^2 - \lambda n_{L}(q_1^2) \in \partial L^{\circ}$ is the first to occur then $p_2 = (\beta' p_1^1 + (1-\beta') p_2^1, p_2^2)$, where $\beta' \in (0, \beta)$ is chosen such that $\beta p_1^1 + (1-\beta)p_2^1 - \lambda n_{K}(q_1^1) = \beta' p_1^1 + (1-\beta') p_2^1$. In terms of the previous proof $p_2$ has the first type.

In both cases, the obtained point $p_2$ indeed lies in the relative interior of the facet of $H \times H^{\circ}$.

The following coordinate change from $q_1$ to $q_2$ can be determined even easier: We know that for some  $\mu > 0$ it must hold
$q_2 - q_1 = (\mu n_{K^{\circ}}(p_2^1), 0_l)$ (if $p_2$ has the first type) or
$q_2 - q_1 = (0_k, \mu n_{L^{\circ}}(p_2^2))$ (if $p_2$ has the second type); this gives
$q_2 = (\alpha q_2^1, (1-\alpha) q_1^2)$ (if $p_2$ has the first type) or
$q_2 = (\alpha q_1^1, (1-\alpha) q_2^2)$ (if $p_2$ has the second type).

Continuing in this fashion we obtain a simple classical billiard trajectory in $H \times H^{\circ}$, satisfying the required properties.
\end{proof}

\begin{corollary}
Let $H$ be a Hanner polytope.
Among all specifications of an initial coordinate $q_1 \in \partial H$ and an initial momentum $p_1 \in \partial H^{\circ}$, satisfying the property that the normal $-n_{H}(q_1)$ put from $p_1$ looks inside $H^{\circ}$, almost all generate a closed simple classical billiard trajectory of length $4$ in the configuration $H \times H^{\circ}$.
\end{corollary}

\begin{proof}
It is easy to see that given an initial position $(q_1,p_1)$ (if these points lie in the relative interiors of the facets of $H$ and $H^{\circ}$) one can uniquely determine $\alpha, \beta \in (0,1)$ from the previous proof and also determine the initial positions in two Hanner configurations of less dimensions. In these two problems of less dimensions induction hypothesis gives us two trajectories, serving as the projections for the trajectory we search for. After that we are in the setting of theorem~\ref{theorem:hannervariety}.
\end{proof}

\begin{corollary}
Let $H$ be a Hanner polytope.
In an arbitrarily small neighborhood of any point $(q,p) \in \partial (H \times H^{\circ})$ there always exists a billiard (in configuration $H \times H^{\circ}$) trajectory of the minimal length.
\end{corollary}

\begin{proof}
We can suppose $p \in \partial H^{\circ}$ without loss of generality. Slightly stepping aside we can make $q$ to lie in the interior of $H$ and make $p$ to lie in the relative interior of a facet of $H^{\circ}$. The ray emanating from $q$ in the direction $n_{H^{\circ}}(p)$ intersects $\partial H$ in a certain point $q'$. According to the previous corollary, $(q',p)$ almost surely can be extended up to a simple classical trajectory; hence in its small neighborhood we can find a good initial position extendible to the required trajectory, passing close to $(q,p)$.
\end{proof}

\section{Billiards in bodies of constant width}
\label{constwid}

In section we deal with the Euclidean case $T=B$ and consider a different problem, related to the previous results.

\begin{theorem}
\label{theorem:constwid}
The minimal billiard trajectory in a body of constant width in the plane is 2-periodic.
\end{theorem}

The proof of~\cite[Theorem~1.2]{bb2009} also suits well for proving this statement.
We quote below another argument, whose essential part is proposed by A.~Zaslavsky.

\begin{proof}
By Theorem~\ref{theorem:bezdeks}, it remains to exclude the case of a 3-periodic trajectory. Suppose, in a body $K \subset \mathbb{R}^2$ of a constant width 1, there is a 3-periodic billiard trajectory of length not exceeding 2. Let $A, B, C$ be the vertices of the trajectory. We draw the segments $AA_1, BB_1, CC_1$ of length 1 along the bisectors of $\triangle ABC$. If we measure the width of $K$ along the direction $AA_1$ constructing both supporting lines orthogonal to $AA_1$ then we must obtain points $A$ and $A_1$ as the tangency points (otherwise the distance between them is more than 1). That's why $A_1 \in K$ and, similarly, $B_1 \in K, C_1 \in K$.

Put $\angle BAC = \alpha, \angle ABC = \beta, |BC| = a, |AC| = b, |AB| = c, p = \frac{a+b+c}{2} \le 1$.
Now we prove even stronger statement: All the distances $|A_1B_1|, |B_1C_1|, |C_1A_1|$ are greater than 1. For $|A_1B_1|$ (other cases are similar), we are going to establish the chain of inequalities:
$$
|A_1B_1|^2 \ge 1 + (1-c)^2 + 2c(1-\cos \frac{\alpha}{2})(1-\cos \frac{\beta}{2}) > 1.
$$

Let $J_A$ be the excenter of $\triangle ABC$ adjacent to $BC$ and let $T_A$ be the base of the perpendicular to the line $(AB)$ through $J_A$. It is elementary that $|AT_A| = p$.

\begin{figure}
\centering
\includegraphics[width=1.0\textwidth]{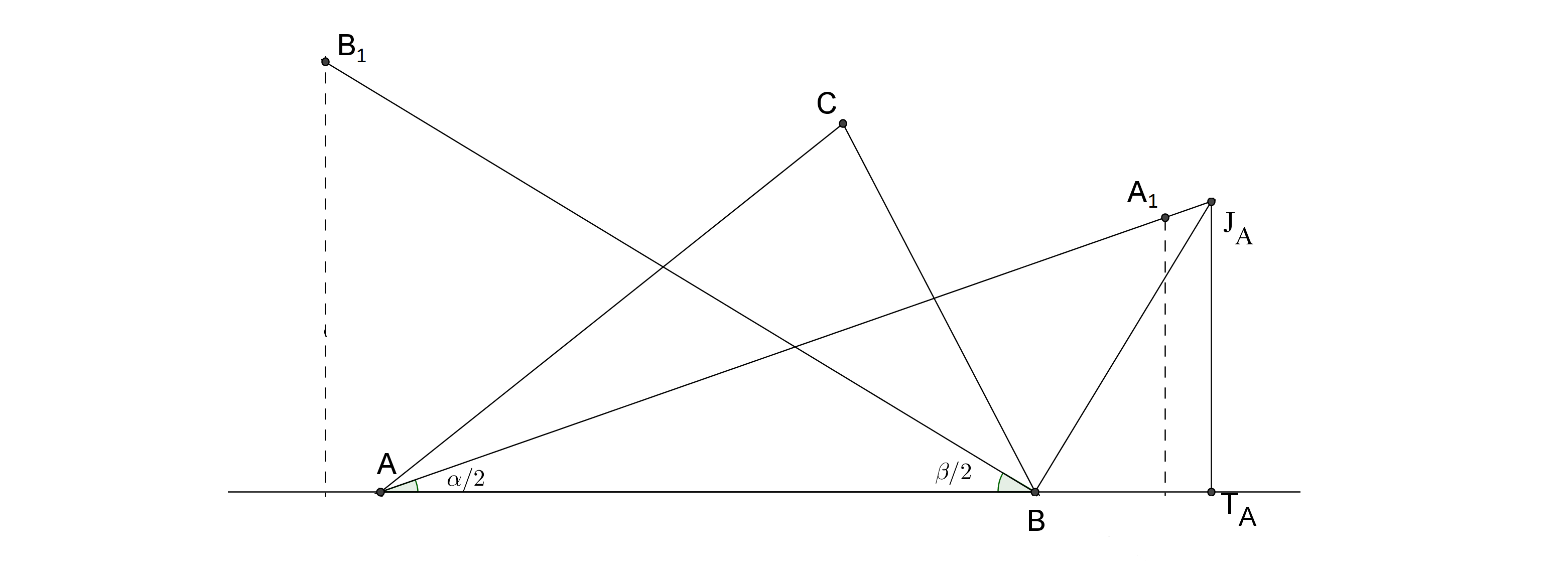}
\caption{Illustration to the proof of theorem~\ref{theorem:constwid}}
\label{picture:zaslavsky}
\end{figure}

Further, $|BT_A| = |J_AT_A| \tan \frac{\beta}{2} = |AT_A| \tan \frac{\alpha}{2} \tan \frac{\beta}{2}$, whence
$$
\frac c p = 1 - \frac{|BT_A|}{|AT_A|} = 1 - \tan \frac{\alpha}{2} \tan \frac{\beta}{2}.
$$
The projection of $A_1B_1$ on the line $(AB)$ is equal to $\cos \frac{\alpha}{2} + \cos \frac{\beta}{2} - c$ and the projection of $A_1B_1$ on the direction orthogonal to $(AB)$ is equal to $|\sin \frac{\alpha}{2} - \sin \frac{\beta}{2}|$; therefore
\begin{multline*}
|A_1B_1|^2 = (\cos \frac{\alpha}{2} + \cos \frac{\beta}{2} - c)^2 + (\sin \frac{\alpha}{2} - \sin \frac{\beta}{2})^2 = \\
= 2 + c^2 - 2c (\cos \frac{\alpha}{2} + \cos \frac{\beta}{2}) + 2 (\cos \frac{\alpha}{2} \cos \frac{\beta}{2} - \sin \frac{\alpha}{2} \sin \frac{\beta}{2}) = \\
= 1 + (1-c)^2 + 2c - 2c (\cos \frac{\alpha}{2} + \cos \frac{\beta}{2}) + 2\frac c p \cos \frac{\alpha}{2} \cos \frac{\beta}{2} \ge \\
\ge 1 + (1-c)^2 + 2c(1-\cos \frac{\alpha}{2})(1-\cos \frac{\beta}{2}),
\end{multline*}
since $p \le 1$.
\end{proof}

\end{document}